\documentclass[11pt]{article}

\usepackage{amssymb,bbm}
\usepackage{amsfonts}
\usepackage{latexsym}
\usepackage{amsmath}
\usepackage{amsthm}
\usepackage{geometry} 
\usepackage{graphicx}
\usepackage{caption}
\usepackage{verbatim}
\usepackage[colorlinks=true]{hyperref} 
\usepackage{lineno}

\usepackage{amsmath,amsthm,amsfonts,amssymb, nccmath}
\usepackage[lofdepth,lotdepth]{subfig}
\newtheorem{proposition}{Proposition}
\newtheorem{definition}{Definition}
\newtheorem*{definition*}{Definition}

\newtheorem{remark}{Remark}

\usepackage{tabularx,booktabs}
\newcolumntype{C}{>{\centering\arraybackslash}X} 
\usepackage{color}
\usepackage{comment}
\usepackage{bigints}
\usepackage[linesnumbered,lined,boxed,commentsnumbered, ruled]{algorithm2e}

\usepackage{tikz,xcolor,hyperref}

\definecolor{lime}{HTML}{A6CE39}
\DeclareRobustCommand{\orcidicon}{
	\begin{tikzpicture}
	\draw[lime, fill=lime] (0,0) 
	circle [radius=0.16] 
	node[white] {{\fontfamily{qag}\selectfont \tiny ID}};
	\draw[white, fill=white] (-0.0625,0.095) 
	circle [radius=0.007];
	\end{tikzpicture}
	\hspace{-2mm}
}

\foreach \x in {A, ..., Z}{\expandafter\xdef\csname orcid\x\endcsname{\noexpand\href{https://orcid.org/\csname orcidauthor\x\endcsname}
			{\noexpand\orcidicon}}
}

\def\lddots{\mathinner{\mkern1mu\raise1pt\hbox{.}\mkern2mu  
\raise4pt\hbox{.}\mkern2mu\raise7pt\vbox{\kern7pt\hbox{.}}\mkern1mu}}
\makeatletter
\def\numberbysection{\@addtoreset{equation}{section}
 \def\theequation{\thesection.\arabic{equation}}}
\makeatother  
  
\numberbysection

\newcommand{\be}{\begin{eqnarray}}  
\newcommand{\ee}{\end{eqnarray}}

 \title{\bf An analytical reconstruction formula with efficient implementation for a modality of Compton Scattering Tomography with translational geometry }
\author{\textsf{Cécilia Tarpau$^{a,b,c, *}$\orcidA{},}
\textsf{Javier Cebeiro$^d$\orcidB{},}
\textsf{Geneviève Rollet$^a$\orcidC{},}\\
\textsf{Mai K. Nguyen$^b$\orcidD{}}
\textsf{and Laurent Dumas$^c$\orcidE{}}
\\
\textit{\footnotesize $^{a}$ LPTM (UMR 8089), CY Cergy Paris Université, CNRS, Cergy-Pontoise, France} 
\\
\textit{\footnotesize $^{b}$ETIS (UMR 8051), CY Cergy Paris Université, ENSEA, CNRS, Cergy-Pontoise, France} 
\\
\textit{\footnotesize $^{c}$ LMV (UMR 8100), Université de Versailles Saint-Quentin, CNRS, Versailles, France}
\\ 
\textit{\footnotesize $^{d}$ CEDEMA, Universidad Nacional de San Mart\'in, San Mart\'in, Argentina}
\\
{\footnotesize $^{*}$ Corresponding author} \\}
\date{}

\clearpage
\begin{document}

\maketitle
\thispagestyle{empty}
\abstract{\footnotesize  In this paper, we address an alternative formulation for the exact inverse formula of the Radon transform on circle arcs arising in a modality of Compton Scattering Tomography in translational geometry proposed by Webber and Miller (Inverse Problems (36)2, 025007, 2020). The original study proposes a first method of reconstruction, using the theory of Volterra integral equations. The numerical realization of such a type of inverse formula may exhibit some difficulties, mainly due to stability issues. Here, we provide a suitable formulation for exact inversion that can be straightforwardly implemented in the Fourier domain. Simulations are carried out to illustrate the efficiency of the proposed reconstruction algorithm.

}

\section{Introduction}
\label{sec:introduction}

Compton Scattering Tomography (CST) is an imaging technique whose objective is to exploit wisely Compton scattered photons by the object to scan in order to reconstruct its electron density map. Since the early proposition of this type of imaging by Lale \cite{lale1959examination}, Clarke \cite{clarke1969compton} and Farmer \cite{farmer1971new}, studies about CST systems proved already promising results in medical imaging, for instance for the identification of lung tumours \cite{redler2018compton, jones2018characterization}, but also for earthquake engineering \cite{gautam83}, cultural heritage imaging \cite{prado2017three, harding2010compton}, landmine detection \cite{hussein2005use} and agricultural measurements \cite{cruvinel2006compton}. In fact, CST has made it possible to widen the fields of application for tomography to the imaging of one-sided large objects,  because, in some configurations, sources and detectors can be placed at the same side of the object. 

The proposition of CST systems is strongly related to the study of the associated integral transform, which models data acquisition \cite{livre, nguyen2006imagerie}. These integral transforms are generalizations of the classical Radon transform on lines studied by Radon \cite{radon17} and Cormack \cite{cormack63}. While, in two-dimensional CST, manifolds are on circle arcs or double circle arcs, in three dimensions, data is acquired on toric surfaces. These circular geometries for the considered manifolds originate from the Compton effect. This physical phenomenon occurs when a photon emitted by the source with energy $E_0$ collides with an electron as it passes through matter. This photon is scattered, and deviated by an angle $\omega$ from its original direction. The photon looses also a part of its energy. The Compton formula gives us the one-to-one correspondence between the energy of scattered photons $E(\omega)$ and the related scattering angle $\omega$ 
\begin{equation}
     E(\omega) = \frac{E_0}{1+\frac{E_0}{m\,c^2}\left(1-\cos(\omega)\right)}.
\end{equation}
$m$ is the electron mass and $c$ the speed of light. This relation ensures also that scattering sites of photons with identical energy $E(\omega)$ are located on a circle arc labelled by the scattering angle $\omega$. Figure \ref{fig:cst_principle} illustrates the general functioning principle of a CST system. 

\begin{figure}
    \centering
    \includegraphics[width = 0.3\textwidth]{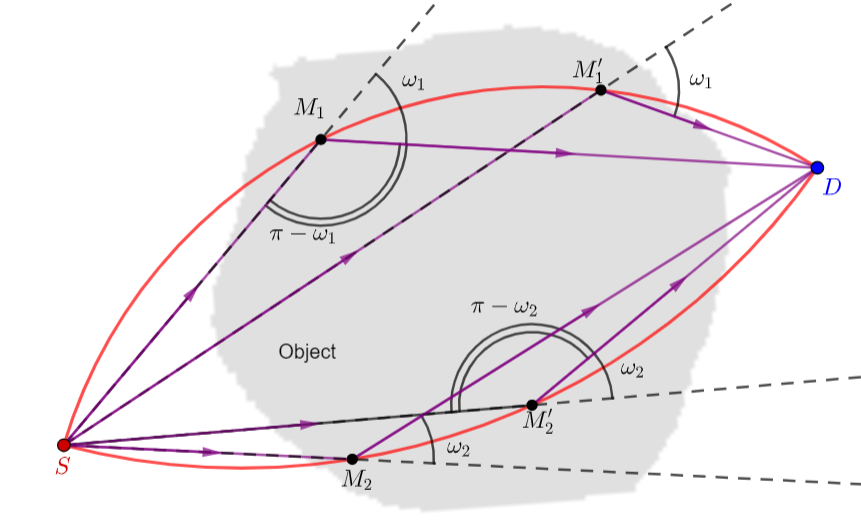}
    \caption{\small General functioning principle of a CST system. Photons are emitted by source $S$, interact at sites $M$, and are recorded at site $D$. When a photon is detected carrying an energy $E(\omega_1)$ (resp. $E(\omega_2)$), the possible interaction sites lie on the upper (resp. lower) circle arc which subtends the angle $(\pi-\omega_1)$ (resp. $(\pi-\omega_2)$).}
    \label{fig:cst_principle}
\end{figure}

Several issues around the study of these generalized Radon transforms are then of interest, such as existence and uniqueness of the solution, its stability or the range conditions. The main important problem remains the reconstruction of the image, mostly coming from the proposition of analytical inverse formulas. 

\begin{figure}[!hb]
    \centering
	\subfloat[]{\includegraphics[width= 0.25\textwidth]{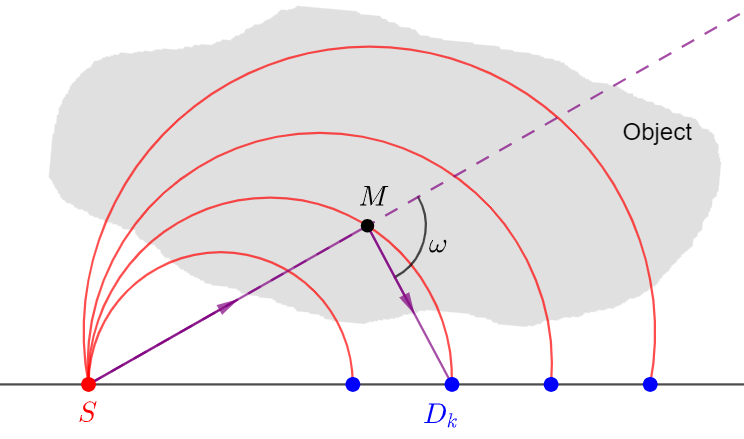}
	\label{fig:norton_modality}
	}\hspace{1cm}
	\subfloat[]{\includegraphics[width= 0.2\textwidth]{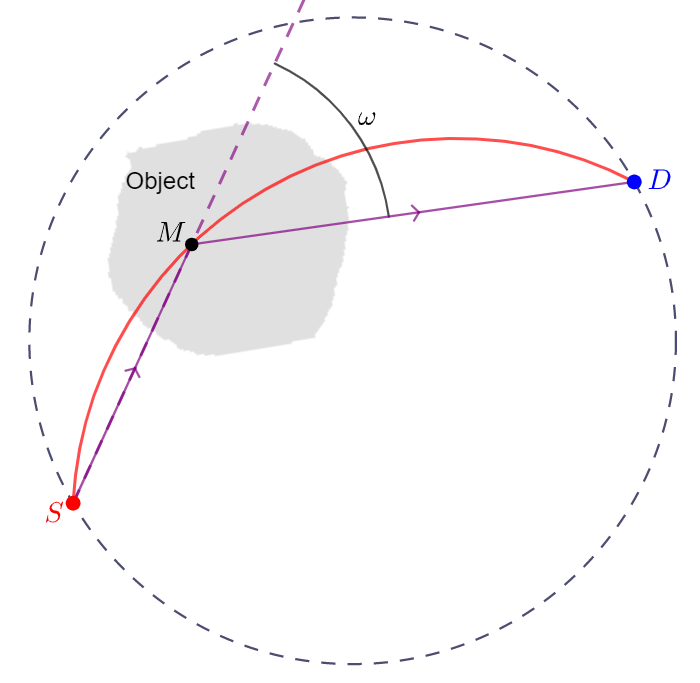}
	\label{fig:nguyen1_modality}}
	\hspace{1cm}
	\subfloat[]{\includegraphics[width= 0.25\textwidth]{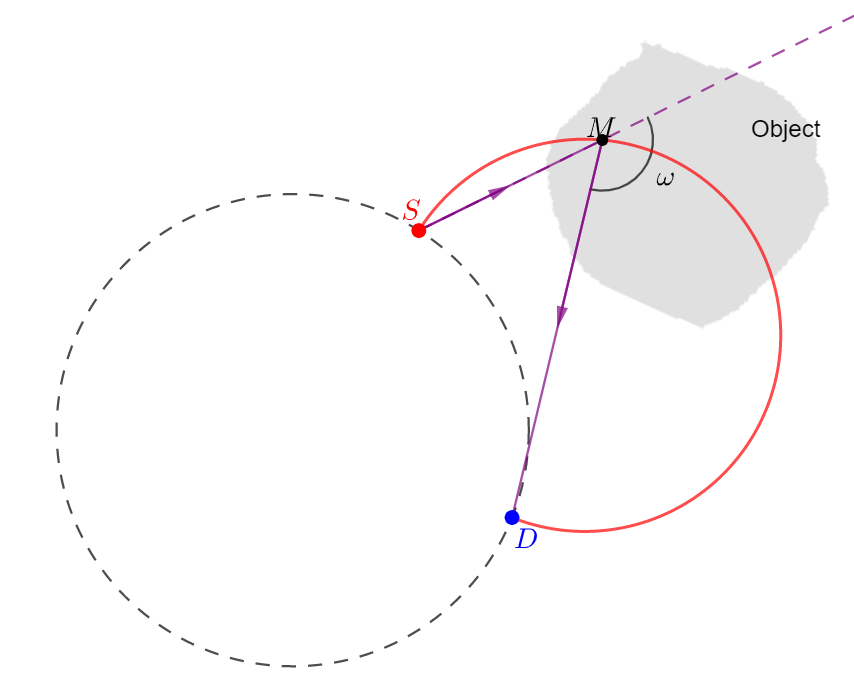}
	\label{fig:nguyen2_modality}}\\
	\subfloat[]{\includegraphics[width= 0.25\textwidth]{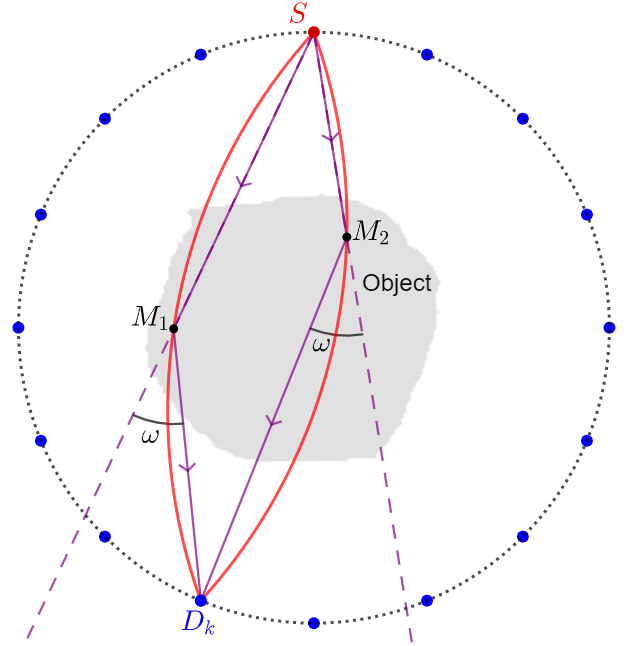}
	\label{fig:ccst}}\hspace{1cm}
	\subfloat[]{\includegraphics[width= 0.25\textwidth]{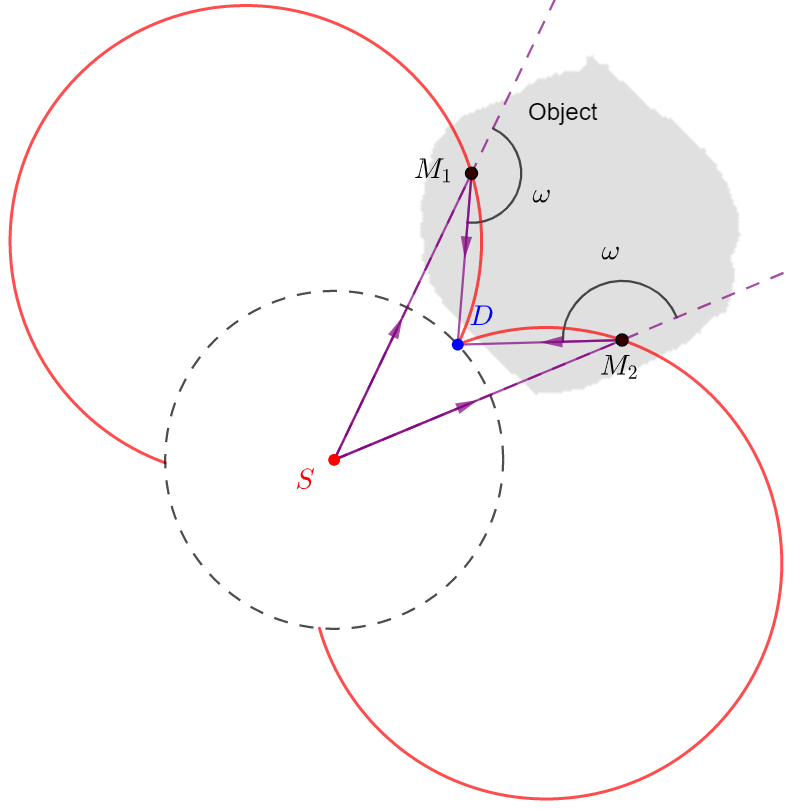}
	\label{fig:cst4}}\hspace{1cm}
	\subfloat[]{\includegraphics[width= 0.25\textwidth]{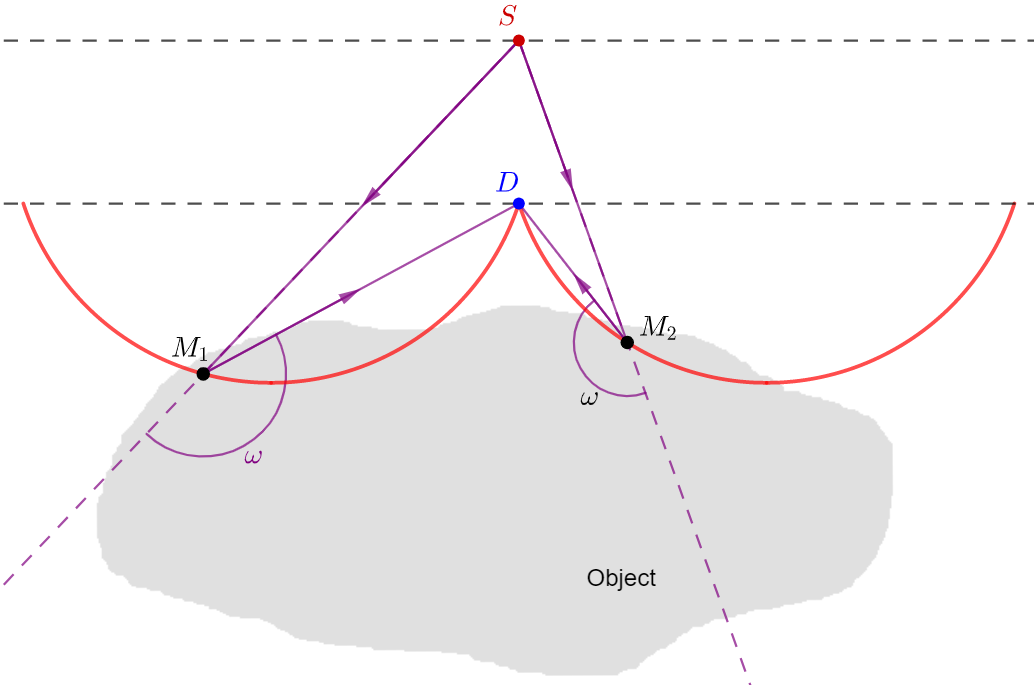}
	\label{fig:cstwebber}}
	\caption[Working principle of CST systems and previous proposed ones]{\small Previous proposed CST modalities (a, b, c, d, e, f). (a): Fixed source and detectors placed on a line. (b): Rotating pair source-detector diametrically opposed. (c): Rotating pair source detector. (d): Fixed source and detectors placed on a ring. (e): Detector rotating around a fixed source. (f) Source-detector translating simultaneously along two parallel lines. In all figures: The source $S$ is represented by a red point. The detector(s) $D$ is (are) represented by blue point(s). The $M$, $M_i$ or $M'_i$ in black, are running points and examples of scattering site. An example of trajectory for a photon whose scattering site is $M$ is shown in purple. The corresponding scattering angle is denoted $\omega$. The object to scan is represented in grey. The red continuous curves are the examples of scanning circles arcs. For (c), (e) and (f), the dashed circles (resp. lines) represents the circular (resp. linear) paths on which move the sensors.}
	\label{fig:previousCSTmodalities}
\end{figure}

In that way, the first studied two-dimensional CST system was made of a fixed line of detectors containing a fixed source \cite{norton94}. See Figure \ref{fig:norton_modality}. The associated scanning circle arcs consist of a family of semicircles with a common fixed end point (the point source) and another extremity (the considered detector) on the line. Then, circular geometries for CST have been proposed, the first one considers a pair source - detector diametrically opposed in rotation around the object \cite{nguyen10, num_inv, rigaudIEEE13} and data acquisition is modelled by a Radon transform on circle arcs having a fixed source-detector chord length. See Figure \ref{fig:nguyen1_modality}. The second consists also of a pair source - detector moving on a circle, but the distance between the two is no longer constant and depends on the scattering angle \cite{truong2011radon} (Fig. \ref{fig:nguyen2_modality}). With this modality, data measurement is performed on a family of circles orthogonal to the fixed circular path of the pair source - detector. In a third modality (see Fig. \ref{fig:ccst}), on the contrary, it was considered a set of detectors  placed on a ring containing the source, thus obtaining a completely fixed CST modality \cite{tarpau19trpms, tarpau2020compton, cebeiro18, rigaud_rotation_free2017}. The corresponding Radon transform measures the contribution of the photons scattered at the points located on circles arcs having a common extremity (the point source) and another (the considered detector) on the detector ring. Some configurations, mentioned above, employ collimators to split up photons coming from different circle arcs. Geometries without collimation have also been studied. 
Consequently, for a given position of detector, the acquisition is performed on a family of double circle arcs and the amount of registered data is potentially doubled. This feature may also reduce the acquisition time and finally radiation exposition in comparison with a similar geometry with collimation. Among the proposed CST systems without collimation, one supposes a fixed source and a detector rotating around this source \cite{tarpau2020tci}. See Figure \ref{fig:cst4}. The other one, proposed in \cite{webber2020compton}, is made of a source and a detector which translate along a line (Fig. \ref{fig:cstwebber}).

The direct reconstruction of volumes is also of interest with the proposition of three-dimensional CST systems. These modalities use uncollimated sources and detectors, and data acquisition is performed on toric surfaces. In many cases, these 3D modalities correspond to extensions of 2D systems. Circular geometries become thus spherical or cylindrical \cite{webber2018three, rigaud20183d, cebeiro2021ip} and linear geometries become planar \cite{webber2020compton}.

Here, we are interested in the two-dimensional modality proposed in \cite{webber2020compton}. The purpose of this article is to develop an alternative formulation suitable for the development of a faster and efficient reconstruction algorithm. The associated reconstruction algorithm will use only classical tools such as Fast Fourier Transform (FFT) algorithm. 

The paper is outlined as follows. Section \ref{sec:modeling} recalls the general setup of the system and the model for data acquisition for the modality. Section \ref{sec:inverse_formula} introduces the main result of the paper, that is, the alternative formulation for the inverse Radon transform on double circle arcs. Section \ref{sec:numerical_formulations} will give the discrete formulation of the forward operator, as well as the proposed strategy to reconstruct the object under study.  Section \ref{sec:simu_results} discusses the obtained simulation results with a study of the influence of some parameters on reconstruction quality.


\section{Setup and measurement model of the CST system under study}
\label{sec:modeling}
\subsection{Setup}
The system under study is made of a source, assumed to be monochromatic, and a detector separated by a fixed distance from each other. The source and the detector move respectively on a horizontal line of equation $z = 3$ and $z=1$. The horizontal position of the pair source-detector is labelled by $x_0$ (see Figure \ref{fig:presentationCSTWebber}). Alternatively, this system may be sketched with fixed lines of sources and detectors that will be used in pair. 
The object, placed below the detector path, is scanned transversely. No collimation is used at the detector, hence the acquisition is performed on a family of double circle arcs (called toric sections in the original publication \cite{webber2020compton})\footnote{We position ourselves in the same frame of study as the original article, with the same working assumptions. Thus, first order Compton scattering is the only source of attenuation for radiation and data acquisition is performed with a pair source detector, assumed to be point-like. These conditions are common in the literature \cite{nguyen10, webber2018three, webber2017microlocal, cebeiro2017new, truong2019compton, tarpau19trpms, tarpau2020tci} and have been already discussed in \cite{norton94, truong2011radon, webber2018three, tarpau2020tci}.}. We parameterize these circle arcs and define the corresponding Radon transform in the next paragraph. 

\begin{figure}[!ht]
    \centering
    \includegraphics[scale = 0.8]{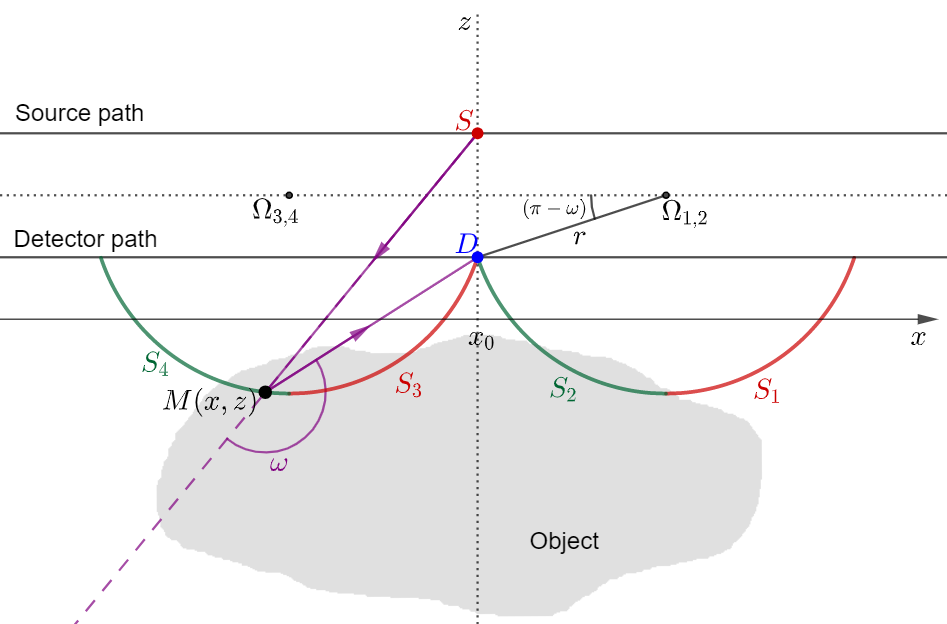}
    \caption{\small Setup and parameterization of the CST modality proposed in \cite{webber2020compton}. The source $S$ and the detector are respectively represented by a red and a blue point. To make the difference between the four half-arcs, $S_1, S_3$ and $S_2, S_4$ are respectively depicted in red and green. $\Omega_{1,2,3,4}$ denote the centres of the circles supporting the half-arcs $S_{1,2,3,4}$. The point $M$ is an example of a scattering site. }
    \label{fig:presentationCSTWebber}
\end{figure}

\subsection{Modelling of data acquisition using the CST system}
Given a scattering angle $\omega$, data acquisition is performed on a family of double circle arcs of radius $r$ where $r = 1/\sin(\pi-\omega)$ (or, equivalently, $\omega = \pi - \arcsin{(1/r)}$). For parameterization, these double circle arcs are obtained with the union of four half arcs denoted $S_j(x_0, r)$, $j\in\{1,2,3,4\}$ of respective equation 
\begin{align*}
    x_1 &= x_0 + \sqrt{r^2-1} + \sqrt{r^2 - (z - 2)^2}, \;\; 
    x_2 = x_0 +\sqrt{r^2-1} - \sqrt{r^2 - (z - 2)^2},\\ 
    x_3 &= x_0 - \sqrt{r^2-1} + \sqrt{r^2 - (z - 2)^2}, \;\; x_4 = x_0 - \sqrt{r^2-1} - \sqrt{r^2 - (z - 2)^2}
\end{align*}
and $z \in ]2-r, 1[$. See Figure \ref{fig:presentationCSTWebber}.

The Radon transform which mathematically models data measurement with this CST system is then defined as follows : 
\begin{definition}
Let $f$ be an unknown function, non-negative, continuous and compactly supported in the half plane $z<1$. The Radon transform on double circle arcs $\mathcal{R}_\mathcal{D}$ maps $f$ into the set of its integrals over the family of double circle arcs as 
\begin{equation}
    \mathcal{R}_\mathcal{D}f(x_0, r) = \int_{\bigcup_{j=1}^{4} S_{j}(x_0, r)} f(x,y) ds.
    \label{eq:R_d_def}
\end{equation}

\noindent where $ds$ refers to the elementary arc length measure on the considered double circle arc. Then, after computation of the arc length measure, we have the explicit reformulation for $\mathcal{R}_\mathcal{D}$ \cite[Proposition 3.1]{webber2020compton} :

\begin{multline}
     \mathcal{R}_\mathcal{D}f(x_0, r) = \int_1^r \frac{1}{\sqrt{1-\left(\frac{z}{r}\right)^2}}\left(\sum_{j=1}^2 f_1\left(\sqrt{r^2-1}+(-1)^j r \sqrt{1-\left(\frac{z}{r}\right)^2}+x_0, z\right) \right.+\\ \left.f_1\left(-\sqrt{r^2-1}+(-1)^j r \sqrt{1-\left(\frac{z}{r}\right)^2}+x_0, z\right)  \right) dz,
     \label{eq:R_d_explicit}
\end{multline}
where $f_1(x,z) = f(x, 2-z)$.
\end{definition}

In the original study of this modality, the invertibility of the corresponding Radon transform as well as its analytical inversion formula has been established. The invertibility was proven using the theory of integral equations and resulted in a Volterra integral equation with a weakly singular kernel in the Fourier domain. This study also leads to a formulation for inversion formula as an integral transformation with a kernel computed iteratively. The numerical calculation of this kind of kernel may require high computational time and/or memory. Furthermore, as mentioned in the Remark 3.4 of the original paper, the proposed approach by Webber is severely ill-posed, particularly in terms of stability. Implementing such a method can lead to large instabilities, even when these are due to small changes in the data. 

\section{An alternative formulation for the inversion formula of the Radon transform on double circle arcs}
\label{sec:inverse_formula}
In this section, we state the main result of the paper, a different formulation for the associated inversion formula, that will be easier to implement numerically.   Let us introduce before some notations that will be used in the proofs. 

\subsection{Notations}

It is useful to define the following transform pairs. 

\begin{definition}[Fourier transform] Let $f$ be a compactly supported function in $\mathbb{R}^n$. The n-dimensional Fourier transform of $f$, denoted $\widehat{f}$, is given by
\begin{equation}
    \widehat{f}(\boldsymbol{\xi}) = \int_{\mathbb{R}^n} f(x)e^{-i\boldsymbol{x}\cdot\boldsymbol{\xi}}d\boldsymbol{x}
\end{equation}
with $\boldsymbol{\xi} \in \mathbb{R}^n$. The inverse Fourier transform is 
\begin{equation}
    f(\boldsymbol{x}) = \frac{1}{(2\pi)^n} \int_{\mathbb{R}^n} \widehat{f}(\boldsymbol{\xi})e^{i\boldsymbol{x}\cdot\boldsymbol{\xi}}d\boldsymbol{\xi}. 
\end{equation}
\end{definition}

\begin{definition}[Fourier cosine transform \cite{Gradshteyn:1702455}] Let $f$ be a compactly supported function in $\mathbb{R}^+$. The Fourier cosine transform of $f$, denoted $\widehat{f}^c$, is given by
\begin{equation}
    \widehat{f}^c(\xi) = \sqrt{\frac{2}{\pi}}\int_{0}^\infty f(x)\cos(x\xi)dx
\end{equation}
with $\xi \in \mathbb{R}$. The inverse Fourier cosine transform is 
\begin{equation}
    f(\xi) = \sqrt{\frac{2}{\pi}}\int_{0}^\infty \widehat{f}^c(\xi)\cos(x\xi)d\xi.
\end{equation}
\end{definition}

We define also the Hankel transform. 

\begin{definition}[Hankel transform \cite{RB1990}] 
Let $f$ be a compactly supported function in $\mathbb{R}^+$. The zero-order Hankel transform of $f$ is defined as
\begin{equation}
    \mathcal{H}_0 f(\eta) = \int_0^\infty f(r) J_0(\eta r)rdr
\end{equation}
where $J_0$ stands for the Bessel function of the first kind of order $0$. 
\end{definition}

Finally, we recall the integral representation of the Bessel function $J_0$:
\begin{equation}
    J_0(x) = \frac{1}{2\pi} \int_{-\pi}^\pi e^{ix\sin{\theta}}d\theta.
\end{equation}

\subsection{Inversion formula}

\begin{proposition}
Denoting $\mathcal{G}f(x_0,r)$ the operator whose Fourier transform according to the first variable is 
\begin{equation}
    \widehat{\mathcal{G}}f(\xi, r) = \frac{\widehat{\mathcal{R}}_\mathcal{D}f(\xi, r)}{2r\cos{(\xi \sqrt{r^2-1})}},
    \label{eq:expressionG}
\end{equation}

\noindent if $r>1$ and $0$ when $r\in[0,1]$, the unknown function $f$ is completely recovered from $\widehat{\mathcal{G}}f$ as follows 
\begin{equation}
    f(x, z) = \frac{1}{4\pi}\int_{-\infty}^{\infty}e^{ix\xi} \int_0^\infty \mathcal{H}_0\widehat{\mathcal{G}}f(\xi, \sqrt{\xi^2+\sigma^2})\cos{(\sigma(2-z))}\sigma d\sigma d\xi.
    \label{eq:final_theorique_post_fourier}
\end{equation}

\end{proposition}

\begin{proof}
With the change of variables $s = z/r$ in \eqref{eq:R_d_explicit} and taking the Fourier transform of $\mathcal{R}_\mathcal{D}$ respectively to variable $x_0$, one gets

\begin{multline}
    \widehat{\mathcal{R}}_\mathcal{D}f(\xi, r) = \int_{-\infty}^\infty dx_0 \int_{1/r}^{1}ds\frac{r e^{-i x_0\xi}}{\sqrt{1-s^2}}\cdot \\\left(\sum_{j=1}^2 f_1\left(\sqrt{r^2-1}+(-1)^j r \sqrt{1-s^2}+x_0, rs\right) + f_1\left(-\sqrt{r^2-1}+(-1)^j r \sqrt{1-s^2}+x_0, rs\right)  \right).
\end{multline}

\noindent With the second change of variables $x=x_0\pm \sqrt{r^2-1}+(-1)^j r\sqrt{1-s^2}$, one gets

\begin{flalign}
    \widehat{\mathcal{R}}_\mathcal{D}f(\xi, r) = 4 \int_{1/r}^{1}ds \frac{r}{\sqrt{1-s^2}} \widehat{f}_1(\xi, rs)\cos{(\xi \sqrt{r^2-1})}\cos{(\xi r \sqrt{1-s^2})}
    \label{eq:TF_R_davec_TF_f1}
\end{flalign}

\noindent where $\widehat{f}_1$ stands for the one-dimensional Fourier transform relatively to the variable $x$.

Using relation \eqref{eq:expressionG},  multiplying both sides of \eqref{eq:TF_R_davec_TF_f1}  by $r\cdot J_0(\eta r)$ with $\eta\geq 1$ and integrating with respect to variable $r$, for $r>1$, one recognizes the Hankel transform of $\widehat{\mathcal{G}f}$, denoted $\mathcal{H}_0\widehat{\mathcal{G}f}$

\begin{equation}
    \mathcal{H}_0\widehat{\mathcal{G}}f(\xi, \eta) = 2\int_1^\infty dr   \int_{1/r}^{1}ds\frac{r}{\sqrt{1-s^2}} \widehat{f}_1(\xi, rs)\cos{(\xi r \sqrt{1-s^2})} J_0(\eta r).
    \label{eq:halkel_of_g}
\end{equation}

\noindent Then, with the double substitution $(r = \sqrt{z^2+b^2}, s = z/\sqrt{z^2+b^2})$, one gets
\begin{equation}
   \mathcal{H}_0\widehat{\mathcal{G}}f(\xi, \eta) = 2 \int_1^\infty dz \widehat{f}_1(\xi, z) \int_{0}^{\infty}db \cos{(\xi b)}  J_0(\eta \sqrt{z^2+b^2}).
\end{equation}

\noindent The result of the $b$-integral is given in the table \cite[p.~55, eq. (35)]{bateman_table_vol1}. Finally, one gets for $0<\xi<\eta$
\begin{equation}
    \mathcal{H}_0\widehat{\mathcal{G}}f(\xi, \eta) = 2\int_1^\infty dz \widehat{f}_1(\xi, z) \frac{1}{\sqrt{\eta^2-\xi^2}}\cos{(z\sqrt{\eta^2-\xi^2})}
\end{equation}
and $\mathcal{H}_0\widehat{\mathcal{G}}f(\xi, \eta) = 0$ if $\eta\leq \xi$. 

\noindent Let $0<\xi<\eta$. Then with the fact $f_1(x, z) = 0$ for $z \in [0,1]$,  

\begin{equation}
     \int_0^\infty dz \widehat{f}_1(\xi, z) \cos{(z\sqrt{\eta^2-\xi^2})} = \frac{\sqrt{\eta^2-\xi^2}}{2} \mathcal{H}_0\widehat{\mathcal{G}}f(\xi, \eta).
     \label{eq:eq_post_table}
\end{equation}

\noindent The left-hand side is the Fourier cosine transform of $\widehat{f}_1(\xi, z)$ according the variable $z$. We can then extract $\widehat{f}_1(\xi, z)$, applying the inverse cosine transform to \eqref{eq:eq_post_table} 
\begin{equation}
    \widehat{f}_1(\xi, z) = \frac{1}{2}\int_0^\infty d\sigma \mathcal{H}_0\widehat{\mathcal{G}}f(\xi, \sqrt{\xi^2+\sigma^2})\cos{(z\sigma)}\sigma
    \label{eq:final_theorique_fourier}
\end{equation}
where $\sigma = \sqrt{\eta^2-\xi^2}$. The final equation is obtained going back to variable $f$, and applying the inverse Fourier transform.
\end{proof}

\begin{remark}
    The projections $\mathcal{\widehat{G}}f$ \eqref{eq:expressionG} contain zeros in the denominator, since the cosine function vanishes when $\xi \sqrt{r^2-1} = 2k\pi \pm \frac{\pi}{2}$, $k\in \mathbb{Z}$. From \eqref{eq:halkel_of_g} to the end of the demonstration, it was supposed that $r$ is different from $\sqrt{1+\left(\frac{\pi}{2\xi}(2k+1)\right)}$. Furthermore, this may be a source of instability in the simulations. The addition of a regularization parameter for simulations is discussed in Section \ref{sec:image_reconstruction_simu} to prevent this.
\end{remark}
\begin{remark}
    Another reconstruction algorithm is also possible from the projections of $\mathcal{G}f$. This process is achieved performing geometric inversion. Geometric inversion is a mapping converting a point $X$ into a point $\tilde{X}$ such that $\tilde{X}X^T = q^2$, where $q \in \mathbb{R}_+^{*}$ is a constant value. The mapped point $\tilde{X}$ has the same direction as the original point $X$ but a distance of $q^2/||X||$ to the origin of the considered coordinate system. As an example, geometric inversion converts circles passing through the origin into straight lines. In the present case, the Radon transform on double circle arcs is converted into a Radon transform on an apparent family of circle arcs of similar geometry as the one studied in \cite{truong2011radon, truong_symmetry}. Although the inverse problem can be alternatively solved using geometric inversion, the approach we employ here is more straightforward. 
\end{remark}

\section{Numerical formulations for the forward and inverse transform}
\label{sec:numerical_formulations}

\subsection{Image formation}

Let $N_{SD}$ be the number of positions for the pair source - detector and $N_r$ the number of double scanning circle arcs per sensor position. We denote $x_{0,k}$, $k \in \{1, ..., N_{SD}\}$ and $r_l$, $l\in\{1, ..., N_r\}$ the discrete variables corresponding respectively to $x_0$ and $r$. The matrix of projection data $\mathcal{R}_\mathcal{D}f(x_{0,k}, r_l)$ is then computed, writing  \eqref{eq:R_d_explicit}  under a discrete form, with the change of variables $z = r\cos{\theta}$

\begin{multline}
    \mathcal{R}_\mathcal{D}f(x_{0,k},r_l) = r_l\, \Delta_\theta\;\cdot \\ \sum_{\theta\in\left[\arcsin{\left(\frac{1}{r_l}\right)}, \frac{\pi}{2}\right]} \left(\sum_{j=1}^2 f_1(x_{0,k} + \sqrt{r_l^2-1}+(-1)^j r_l \cos{\theta}, r_l\sin{\theta}) + f_1(x_{0,k} - \sqrt{r_l^2-1}-(-1)^j r_l\cos{\theta}, r_l\sin{\theta}) \right),
    \label{eq:image_formation}
\end{multline}

where $\Delta_\theta$ is the sampling angular distance of $\theta$.
The above Cartesian parameterization allows having a constant distance between running points of the considered scanning circle arcs during simulations.

\subsection{Image reconstruction }
\label{sec:image_reconstruction_simu}
For image reconstruction, we need to compute the projections $\mathcal{G}f$ in the Fourier domain according to \eqref{eq:expressionG}. This expression contains zeros in the denominator. This may induce instabilities on reconstruction. For simulations, we add a small regularization parameter denoted $\epsilon$
\begin{equation}
   \widehat{\mathcal{G}}f(\xi, r) =  \frac{\widehat{\mathcal{R}}_\mathcal{D}f(\xi, r)}{2r}\frac{\cos{(\xi \sqrt{r^2-1})}}{\epsilon^2+\cos{(\xi \sqrt{r^2-1})}^2}. 
   \label{eq:expressionG_epsilon}
\end{equation}
In terms of computational cost, the most demanding step in the implementation of \eqref{eq:final_theorique_post_fourier} is the calculation of $\mathcal{H}_0\widehat{\mathcal{G}}f$. The idea is to establish a relation between the above operator with the Fourier transform of $\mathcal{G}^\ddag f$, defined as follows
\begin{equation}
    \mathcal{G}^\ddag f(x,z) = \int_{-\infty}^\infty   \mathcal{G}f(x_0,\sqrt{(x-x_0)^2+z^2})\,dx_0. 
    \label{eq:T1dag}
\end{equation}

\noindent We have now the following proposition. 
\begin{proposition} 
 Let $(\sigma, \xi)\in[0, \infty[\times\mathbb{R}$.  $\mathcal{H}_0\widehat{\mathcal{G}}f(\xi, \sqrt{\xi^2+\sigma^2})$ is related with the two-dimensional Fourier transform of the operator $\mathcal{G}^\ddag$ as 
 \begin{equation}
     \mathcal{H}_0\widehat{\mathcal{G}}f(\xi, \sqrt{\xi^2+\sigma^2}) = 2\pi \widehat{\mathcal{G}}^\ddag f(\xi,\sigma). 
     \label{eq:relationH0G_FF2}
 \end{equation}
Consequently, from the inversion formula \eqref{eq:final_theorique_post_fourier}, it follows in the Fourier domain 
\begin{equation}
    \widehat{f}_1(\xi, \sigma) = 2\pi^2|\sigma|\widehat{\mathcal{G}}^\ddag f(\xi,\sigma). 
    \label{eq:relation_ff2}
\end{equation}
where $\widehat{f}_1$ is the two-dimensional Fourier transform of $f_1$. 
\end{proposition}

\begin{proof} From the definitions of Fourier and Hankel transforms and with the integral representation of the Bessel function,
one gets

\begin{align}
    \mathcal{H}_0\widehat{\mathcal{G}}f(\xi, \sqrt{\xi^2+\sigma^2}) &=  \frac{1}{2\pi}\int_{-\infty}^\infty dx_0 \int_0^\infty dr\, \mathcal{G}f(x_0,r) \left(\int_{-\pi}^\pi e^{ir\sqrt{\xi^2+\sigma^2}\sin(\theta)}d\theta\right)\,e^{-i\xi x_0}. 
    \label{eq:reecriture_H_0_G}
\end{align}

There is an angle $\phi{(\sigma, \xi)}\in[0, 2\pi[$ which corresponds to the angular coordinate of the point $(\sigma, \xi)$, such that $\sigma = \sqrt{\xi^2+\sigma^2}\cos{(\phi{(\sigma, \xi)})}$ and $\xi = \sqrt{\xi^2+\sigma^2}\sin{(\phi{(\sigma, \xi)})}$. Using the property of periodicity of trigonometric functions, it follows that
\begin{equation}
    \mathcal{H}_0\widehat{\mathcal{G}}f(\xi, \sqrt{\xi^2+\sigma^2}) =  \frac{1}{2\pi}\int_{-\infty}^\infty dx_0 \int_0^\infty dr\, \mathcal{G}f(x_0,r) \left(\int_{-\pi}^\pi e^{ir\xi\sin{\theta} + \sigma\cos{\theta}}d\theta\right)\,e^{-i\xi x_0}. 
    \label{eq:reecriture_H_0_G_trigo}
\end{equation}

Changing variables $x = r\cos{\theta}$ and $z =r\sin{\theta}$, 
\begin{equation}
    \mathcal{H}_0\widehat{\mathcal{G}}f(\xi, \sqrt{\xi^2+\sigma^2}) =\frac{1}{2\pi}\int_{\mathbb{R}^2} dx dz  \left(\int_{-\infty}^\infty dx_0  \mathcal{G}f(x_0,\sqrt{(x-x_0)^2+z^2}) \right) e^{-i(x\xi+z\sigma)}
    \label{eq:H_O_G_cartésien}
\end{equation}

The right-hand side of \eqref{eq:H_O_G_cartésien} is the two-dimensional Fourier transform of $ \mathcal{G}^\ddag$, weighted by $2\pi$ \eqref{eq:relationH0G_FF2}. We are now able to reformulate the inversion formula \eqref{eq:final_theorique_fourier} as 
\begin{equation}
    \widehat{f}_1(\xi, z) = 2\pi\int_{-\infty}^\infty d\sigma \widehat{\mathcal{G}}^\ddag f(\xi, \sigma)e^{iz\sigma}|\sigma|. 
    \label{eq:final_ant}
\end{equation}

Taking the Fourier transform according to variable $z$ to the above equation \eqref{eq:final_ant} leads to \eqref{eq:relation_ff2}.
\end{proof}

This leads to the reconstruction algorithm summed up in Algorithm \ref{algo:reconssss}. 

\IncMargin{1em}
    \begin{algorithm}[!ht]
    \footnotesize
        \SetAlgoLined
        \KwData{$\mathcal{R}_\mathcal{D}f(x_0, r)$, projections on double circular arcs of function $f$}
        \KwResult{$f(x,y)$}
        Compute the one-dimensional Fourier transform of $\mathcal{R}_\mathcal{D}f(x_0, r)$ using FFT relative to the first variable \;
        Compute $\widehat{\mathcal{G}}f(\xi, r)$ according to \eqref{eq:expressionG_epsilon} and perform the inverse FFT to recover $\mathcal{G}f(x_0,r)$\; 
        For each $x_0$, interpolate the obtained data and sum on all values of $x_0$ to have the back-projected data $\mathcal{G}^\ddag f(x,z)$\;
        Perform the 2D FFT of $\mathcal{G}^\ddag f(x,z)$ and weight by $2\pi^2|\sigma|$\; 
        Compute the inverse FFT of the result to recover $f$ \;
    \caption{\footnotesize Reconstruction of object $f$}
    \label{algo:reconssss}
    \end{algorithm}
    \DecMargin{1em}

\section{Simulations results}
\label{sec:simu_results}

The original object used for simulations is Derenzo phantom, an object made of multiple circles of different sizes. The circles in the object also allow the study of the performance of the algorithm in front of different contrasts and spatial resolution, as well as its ability to reconstruct features locally tangent to lines of any slope.  
The unit length used here is the pixel. We suppose thus that the distance between the source and the detector paths of the modality is two pixels.  
The size of the object is $N\times N = 256\times 256$ pixels in all simulations. Furthermore, given the linear geometry of this modality, there is no loss of generality to consider the object centred relatively to the $z$-axis. 

\subsection{Data acquisition}
Data measurement is calculated according to \eqref{eq:image_formation}. Figure \ref{fig:obj_original&data_acqui} shows an example of data obtained for Derenzo phantom. 

\begin{figure*}[!ht]
      \begin{center}
     \subfloat[]{
      \centering
    \includegraphics[scale = 0.25]{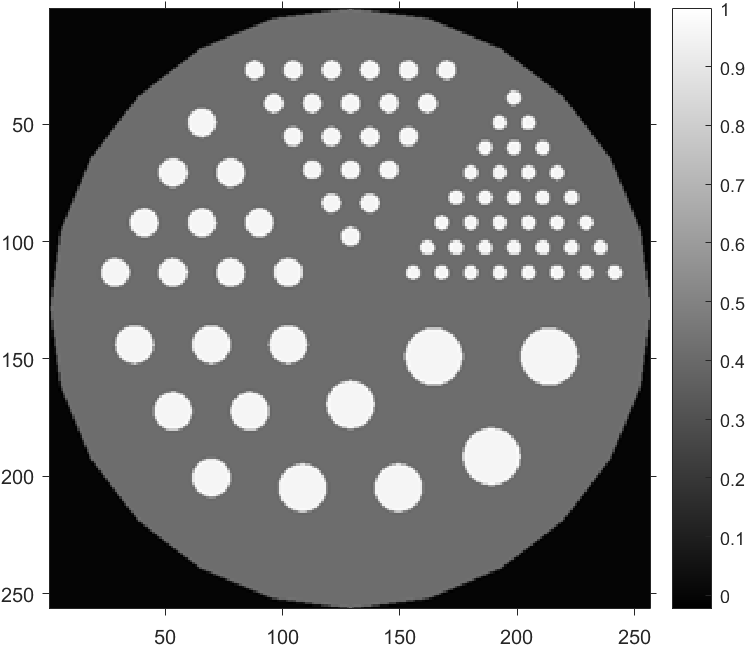}
    \label{fig:derenzo_phantom}
                         }\hspace{0.5cm}
    \subfloat[]{
      \centering
    \includegraphics[scale = 0.23]{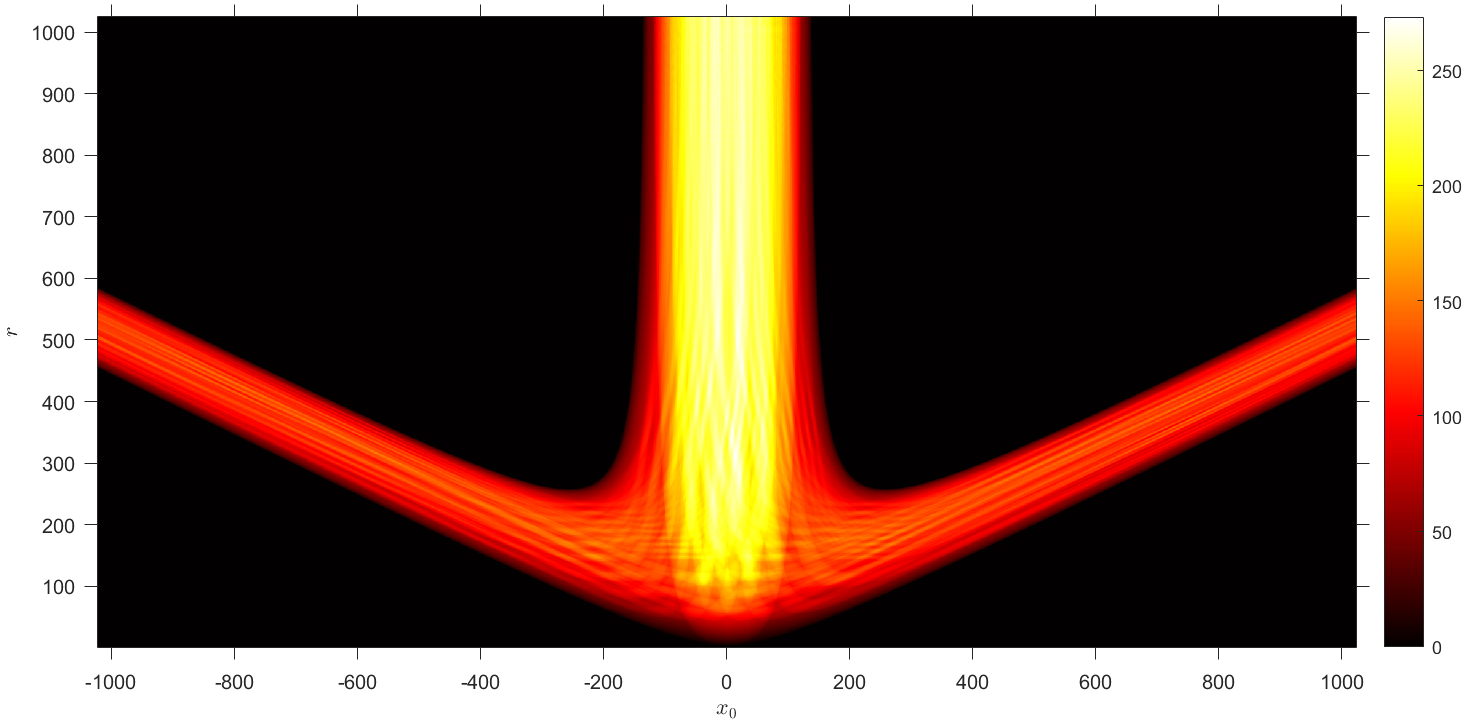}
    \label{fig:acq_Derenzo}
                         } 
    \end{center}
    \caption{(a) Original object: Derenzo phantom. (b) Corresponding acquired data for $N_{SD} = 2048$ and $N_r = 1024$. A distance of one pixel is left between the upper part of the image and the detector path ($\delta= 1$, see \ref{sec:position_obj}).}
  \label{fig:obj_original&data_acqui}
\end{figure*}

\subsection{Influence of some parameters on reconstruction quality}
In the following paragraphs, we study the influence of the other general parameters of the system such as the position of the object, the number of required positions for sensors or the number of scanning circle arcs. In addition to a visual comparison of the reconstruction quality, we propose here to measure quantitatively the error rate between the original object $f_0$ and the reconstruction $f$ with the Normalized Mean Squared Error (NMSE) $=||f-f_0||_2^2/N^2$ where $||.||_2$ refers to the $2$-norm. 

The regularization parameter $\epsilon$, which has to be small, was arbitrarily set to $0.01$.

\subsubsection{Position of the object relative to the detector path}
\label{sec:position_obj}
We analysed here the influence of the position of the object on reconstruction quality. Firstly, the number of positions $N_{SD}$ for the pair source - detector and the number of scanning circles per position $N_r$  was chosen to largely satisfy the well-known condition \cite{RB1990} $N_{SD}\times N_r \geq N^2$ and are set arbitrarily to $N_{SD}=1024$ and $N_r = 1024$. A convenient choice for these parameters will be discussed later. 
We performed various acquisition, modifying the gap $\delta$ between the detector path and the upper part of the object (see Figure \ref{fig:presentationCSTWebber}). Consequently, the object is in the square of Cartesian coordinates $\left(x\in\left[-\frac{N}{2}+1, \frac{N}{2}\right]; z\in[-N-\delta,-\delta+1 \right])$. Figure \ref{fig:varying_delta} shows the reconstruction results for $\delta = 1, 26$ and $51$ pixels which correspond respectively a position for the object in the respective domains $[-N-1, 0]$, $[-N-26, -25]$ and $[-N-51, -50]$ along the z-axis.
The difference between the three reconstructions is on the top of the object. If the object is close to the line of movement of the detector, then this part is less well reconstructed. Indeed, this distance between the object and the detector path allows having arcs of circle tangent horizontally to this part of the object. An offset of $\delta = 51$ seems to be a good trade-off between quality of reconstruction and the applicability of such a measure in practical use. For the rest of the simulations, $\delta$ is set to $51$. 

\begin{figure*}[!ht]
      \begin{center}
     \subfloat[Reconstruction for $\delta = 1$\\NMSE = 0.0112]{
      \centering
    \includegraphics[scale = 0.22]{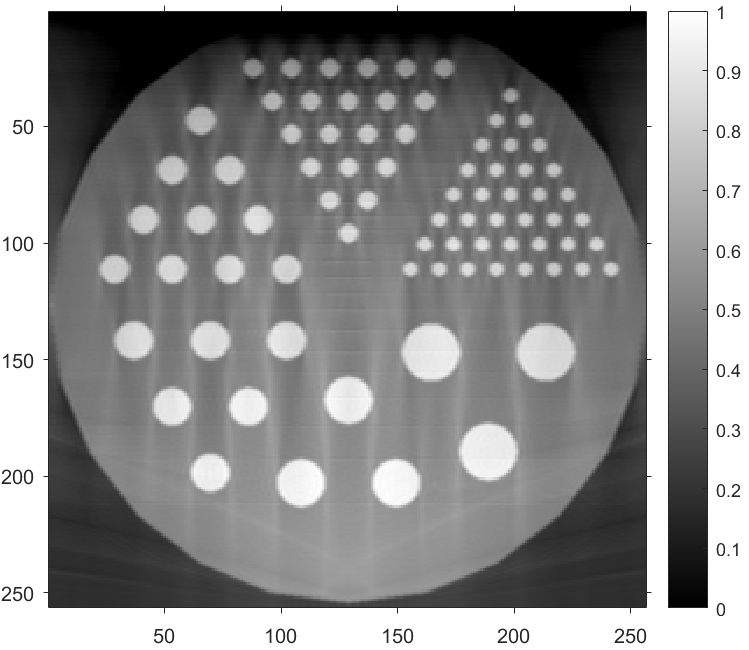}
    \label{fig:der_delta1}
                         } \hspace{0.5cm}
    \subfloat[Reconstruction for $\delta = 26$\\NMSE = 0.0074]{
      \centering
    \includegraphics[scale = 0.22]{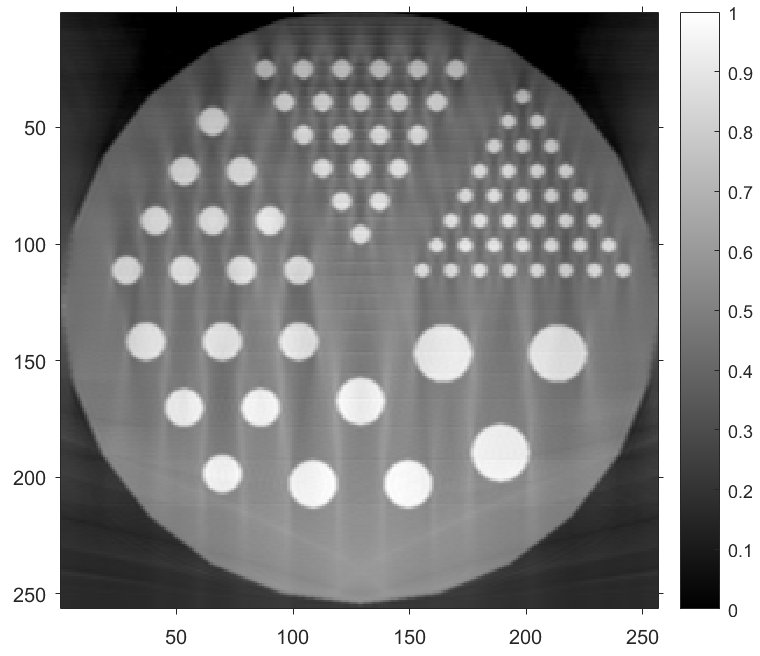}
    \label{fig:der_delta26}
                         }   \hspace{0.5cm}
    \subfloat[Reconstruction for $\delta = 51$\\NMSE = 0.0061]{
      \centering
    \includegraphics[scale = 0.22]{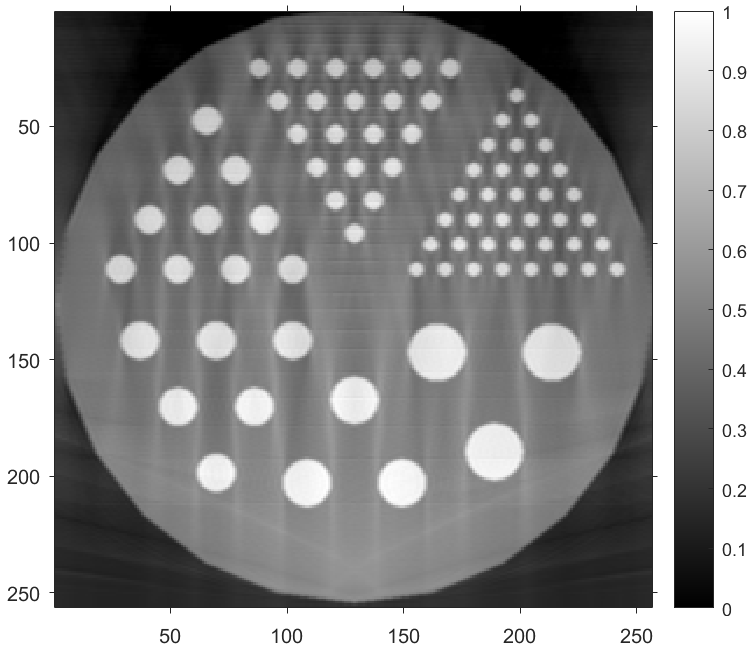}
    \label{fig:der_delta51}
                         }
    \end{center}
    \caption{Reconstruction results of the Derenzo phantom \ref{fig:derenzo_phantom} for $\delta = 1$ (a), $\delta = 26$ (b) and $\delta = 51$ (c) pixel(s).}
  \label{fig:varying_delta}
\end{figure*}

\subsubsection{Number of necessary positions for the pair source-detector}
We studied then the number of different positions required for a good quality of reconstruction. The influence of two running parameters is analysed, first, the farthest position $x_{0, max}$ from the object for the source-detector pair (that is an array of length $[-x_{0,max}, x_{0,max}]$ for the source and detector paths) and the distance $\Delta_{x_0}$ between two adjacent positions of the pair. 

Figures \ref{fig:der_NSD_2N}, \ref{fig:der_NSD_3N} and \ref{fig:der_NSD_4N} show the reconstruction results when the farthest position from the object to the pair source detector is respectively $2N, 3N$ and $4N$ with a common $\Delta_{x_0}$ set to $1$. Reconstruction from a domain $[-x_{0,max}, x_{0,max}] = [-2N, 2N]$ appears to be blurred with strong artefacts in the upper parts of the image. For $x_{0,max} = 3N$ and $4N$, reconstruction quality seems to be visually equivalent, even if the NMSE for $x_{0,max} = 4N$ is higher. This may be due to numerical approximations. For the rest of the simulation, $x_{0,max}$  is set to $3N$. 

The influence of the distance $\Delta_{x_0}$ between two adjacent positions of the pair is now evaluated. Figures \ref{fig:der_nb_d_0.5_per_unit_length}, \ref{fig:der_nb_d_1_per_unit_length} and \ref{fig:der_nb_d_2_per_unit_length} show the result for $\Delta_{x_0}= 2, 1$ and $0.5$ pixels. This represents a respective amount of $0.5, 1$ and $2$ detectors per unit length. In Figure \ref{fig:der_nb_d_0.5_per_unit_length} $(\Delta_{x_0}= 2)$, we can see streaks suggesting a lack of data for reconstructing the object. On the contrary, the doubling of the number of detectors between Fig. \ref{fig:der_nb_d_1_per_unit_length} and Fig. \ref{fig:der_nb_d_2_per_unit_length} does not bring a better quality of reconstruction. Consequently, the use of one detector per unit length seems to be a good trade-off, and the value will remain constant in the rest of the paper. 

\begin{figure*}[!ht]
      \begin{center}
     \subfloat[Reconstruction for $x_{0,max} = 2N$. NMSE = 0.0084]{
      \centering
    \includegraphics[scale = 0.22]{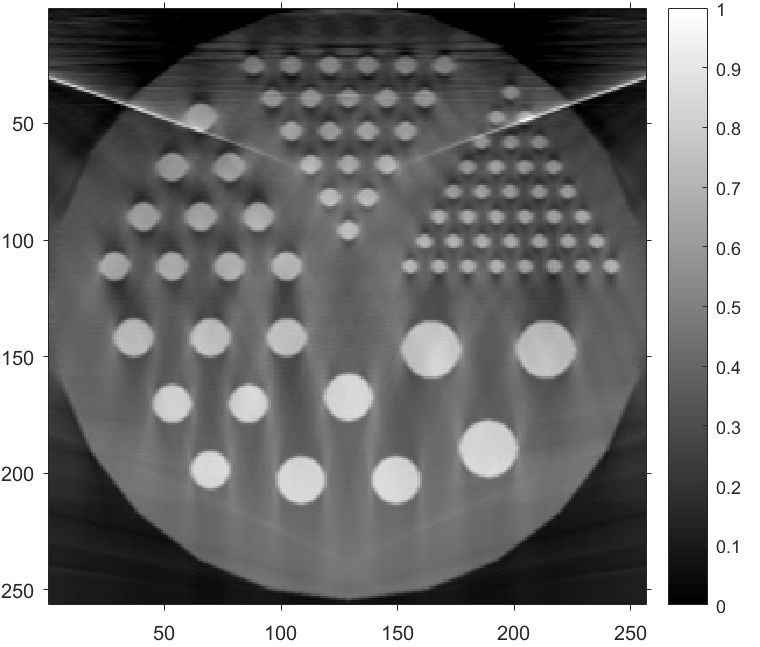}
    \label{fig:der_NSD_2N}
                         } \hspace{0.5cm}
    \subfloat[Reconstruction for $x_{0,max} = 3N$. NMSE = 0.0049]{
      \centering
    \includegraphics[scale = 0.22]{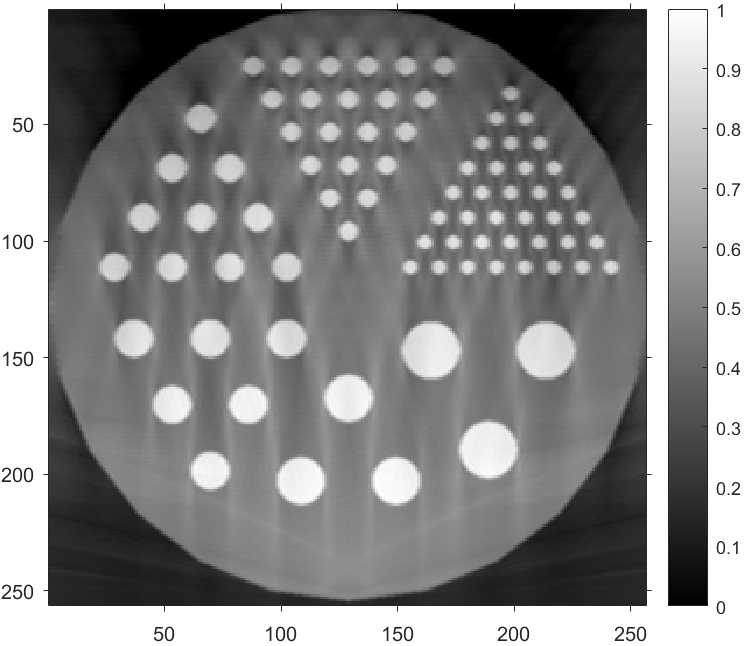}
    \label{fig:der_NSD_3N}
                         }   \hspace{0.5cm}
    \subfloat[Reconstruction for $x_{0,max} = 4N$. NMSE = 0.0061]{
      \centering
    \includegraphics[scale = 0.22]{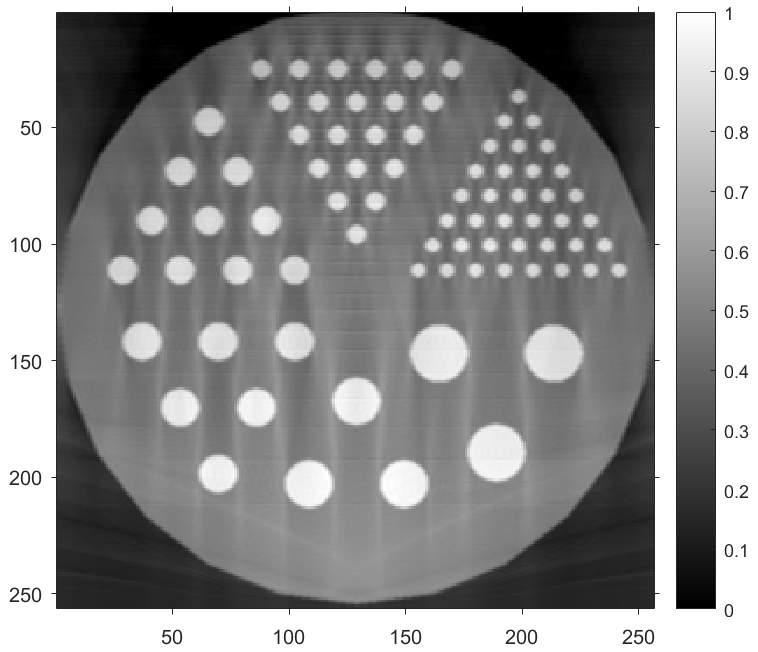}
    \label{fig:der_NSD_4N}
                         } \\
      \subfloat[Reconstruction for $0.5$ detectors per unit length. NMSE = 0.0046]{
      \centering
    \includegraphics[scale = 0.22]{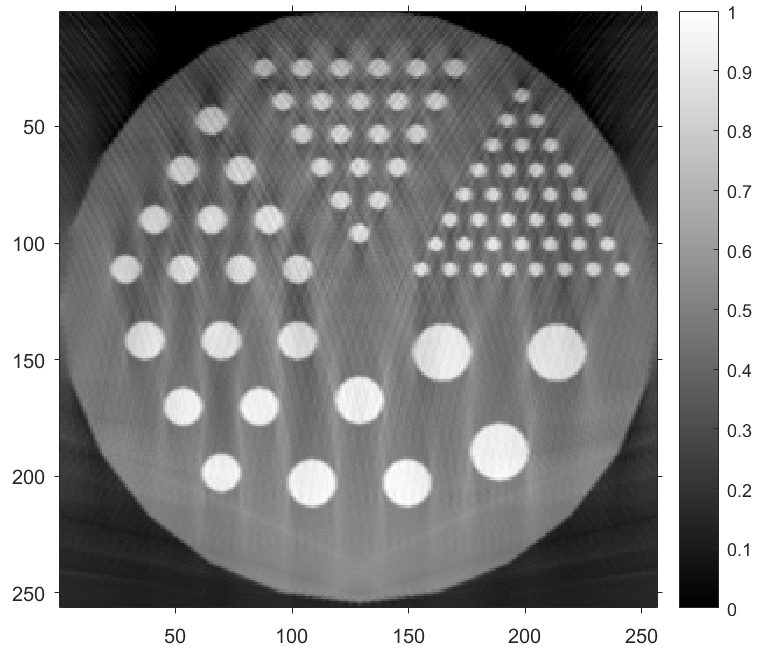}
    \label{fig:der_nb_d_0.5_per_unit_length}
                         } \hspace{0.5cm}
    \subfloat[Reconstruction for $1$ detectors per unit length. NMSE = 0.0049]{
      \centering
    \includegraphics[scale = 0.22]{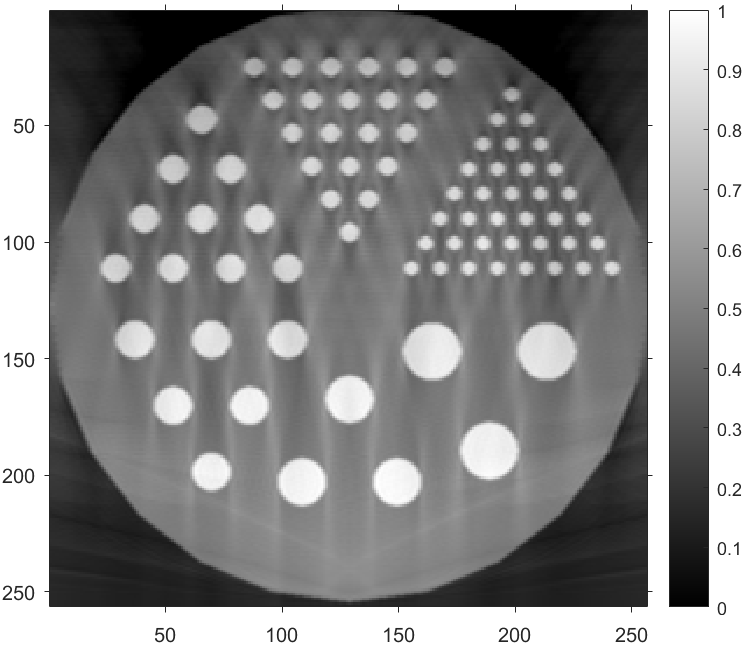}
    \label{fig:der_nb_d_1_per_unit_length}
                         }   \hspace{0.5cm}
    \subfloat[Reconstruction for $2$ detectors per unit length. NMSE = 0.0049]{
      \centering
    \includegraphics[scale = 0.22]{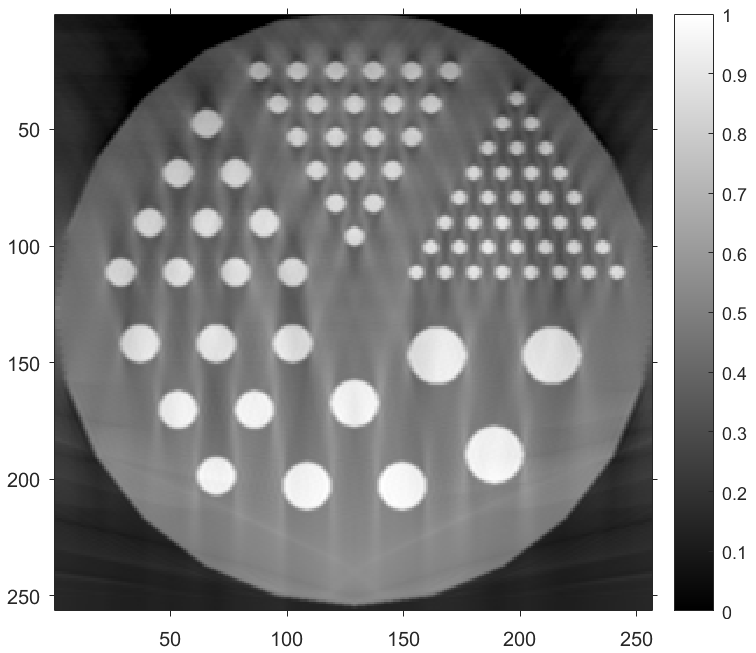}
    \label{fig:der_nb_d_2_per_unit_length}
                         } 
    \end{center}
    \caption{Evaluation of the number of source-detector positions on reconstruction quality. First row: Reconstruction results of the Derenzo phantom \ref{fig:derenzo_phantom} for $x_{0, max} = 2N$ \ref{fig:der_NSD_2N}, $3N$ \ref{fig:der_NSD_3N} and $4N$ \ref{fig:der_NSD_4N} where $\Delta_{x_0}=1$. Second row: Reconstruction results for $0.5$ \ref{fig:der_nb_d_0.5_per_unit_length}, $1$ \ref{fig:der_nb_d_1_per_unit_length} and $0.5$ \ref{fig:der_nb_d_2_per_unit_length} detector per unit length and $x_{0, max} = 3N$ remains constant.} 
  \label{fig:varying_NSD_and_nb_d_per_unitlength}
\end{figure*}

\subsubsection{Number of scanning circles per position of the pair source-detector}
The number of scanning circles necessary for reconstruction is now under study. In the same way, two parameters are of interest, that is, the maximum radius $r_{max}$ of the scanning double circle arcs to be taken and the discretization step $\Delta_r$ that have to be chosen. 
We first evaluate the consequences of the value of $r_{max}$ with three examples on Fig. \ref{fig:der_rmax_2N}, \ref{fig:der_rmax_3N} and \ref{fig:der_rmax_4N} where $r_{max}$ is set respectively to $2N, 3N$ and $4N$ and $\Delta_r = 1$. For $r_{max}=2N$, the reconstruction suggests a lack of data in front of the obtained results for $r_{max}=3N$ and $r_{max}=4N$. Notice the higher NMSE for $r_{max}=4N$, probably due to numerical approximations. 

We were finally looking for the appropriate discretization step for $r$, setting $\Delta r$ to $1, 2$ and $4$. Reconstruction results are shown respectively in Fig. \ref{fig:der_delta_r_1}, \ref{fig:der_delta_r_2} and \ref{fig:der_delta_r_4}. Reconstructions with a large discretization step exhibit blur.  

\begin{figure*}[!ht]
      \begin{center}
     \subfloat[Reconstruction for $r_{max} = 2N$. NMSE = 0.0058]{
      \centering
    \includegraphics[scale = 0.22]{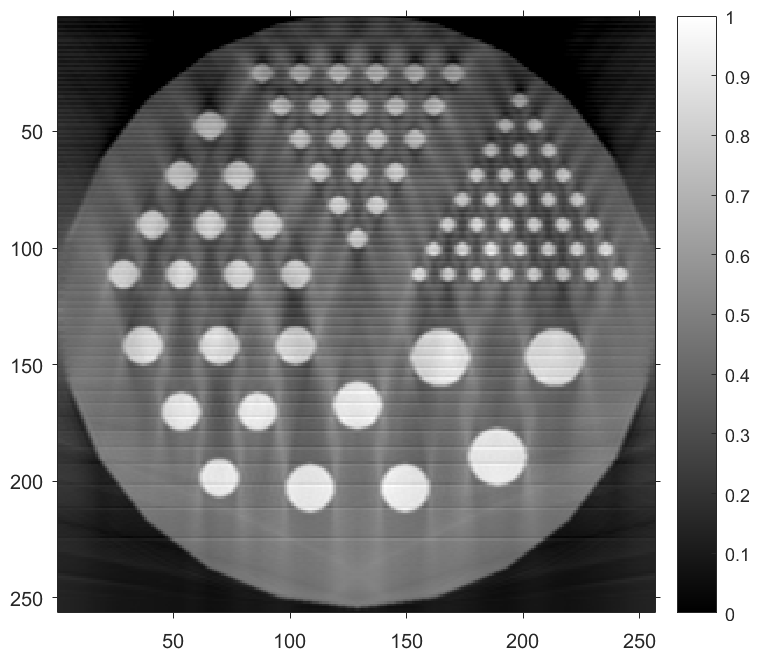}
    \label{fig:der_rmax_2N}
                         } \hspace{0.5cm}
    \subfloat[Reconstruction for $r_{max} = 3N$. NMSE = 0.0040]{
      \centering
    \includegraphics[scale = 0.22]{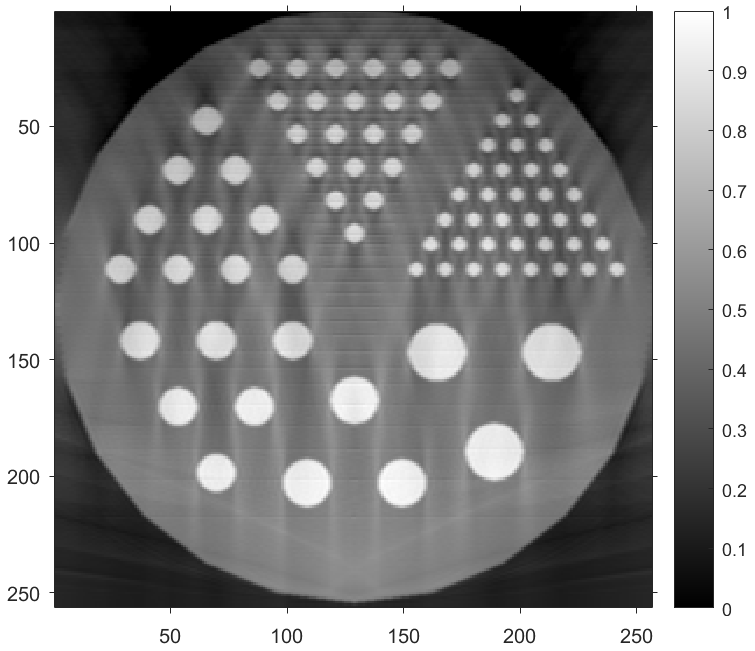}
    \label{fig:der_rmax_3N}
                         }   \hspace{0.5cm}
    \subfloat[Reconstruction for $r_{max} = 4N$. NMSE = 0.0049]{
      \centering
    \includegraphics[scale = 0.22]{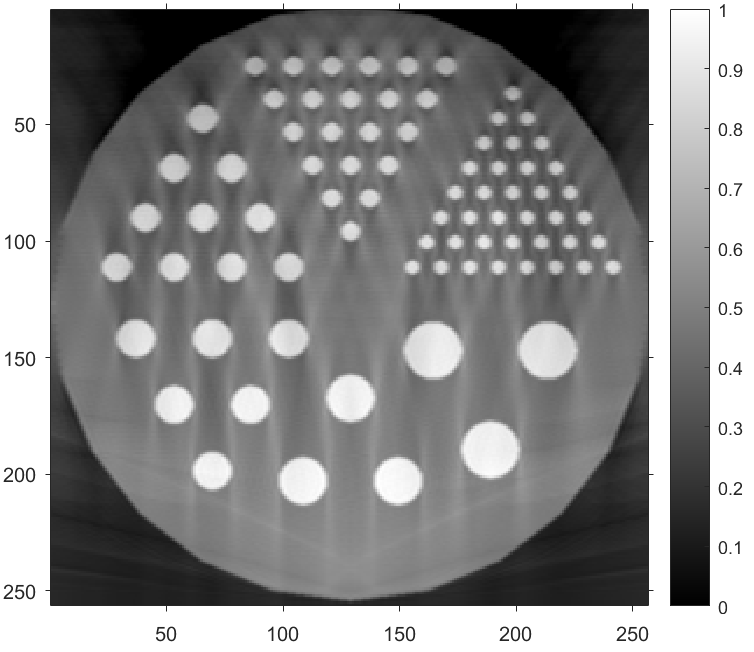}
    \label{fig:der_rmax_4N}
                         } \\
      \subfloat[Reconstruction for $\Delta_r = 1$. NMSE = 0.0040]{
      \centering
    \includegraphics[scale = 0.22]{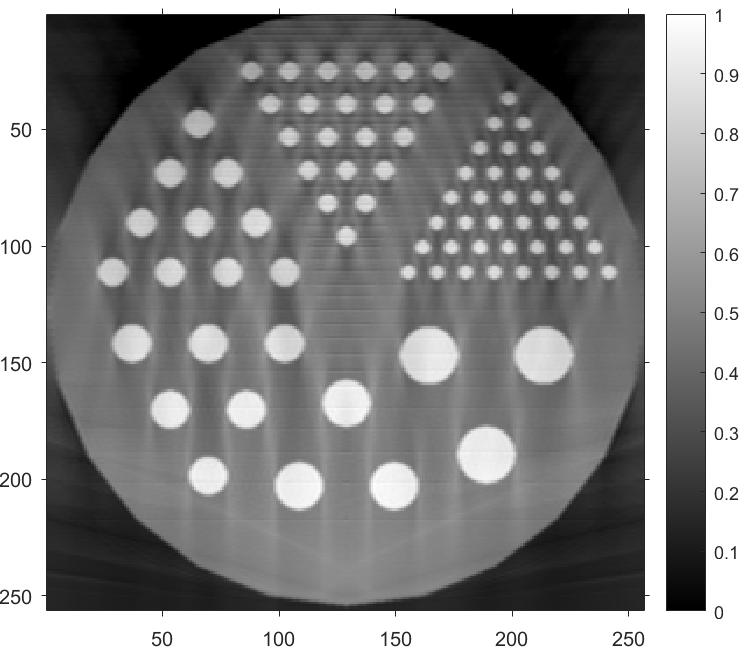}
    \label{fig:der_delta_r_1}
                         } \hspace{0.5cm}
    \subfloat[Reconstruction for $\Delta_r = 2$. NMSE = 0.0043]{
      \centering
    \includegraphics[scale = 0.22]{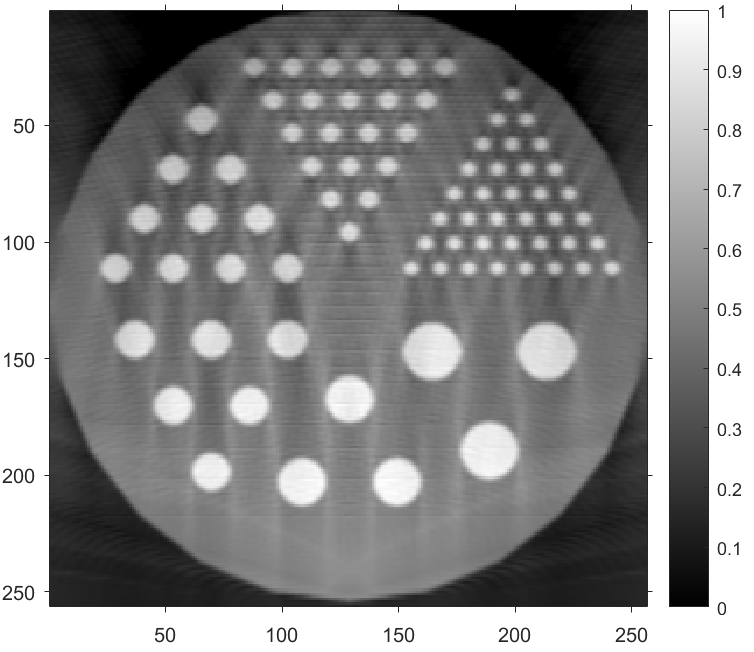}
    \label{fig:der_delta_r_2}
                         }   \hspace{0.5cm}
    \subfloat[Reconstruction for $\Delta_r = 4$. NMSE = 0.0043]{
      \centering
    \includegraphics[scale = 0.22]{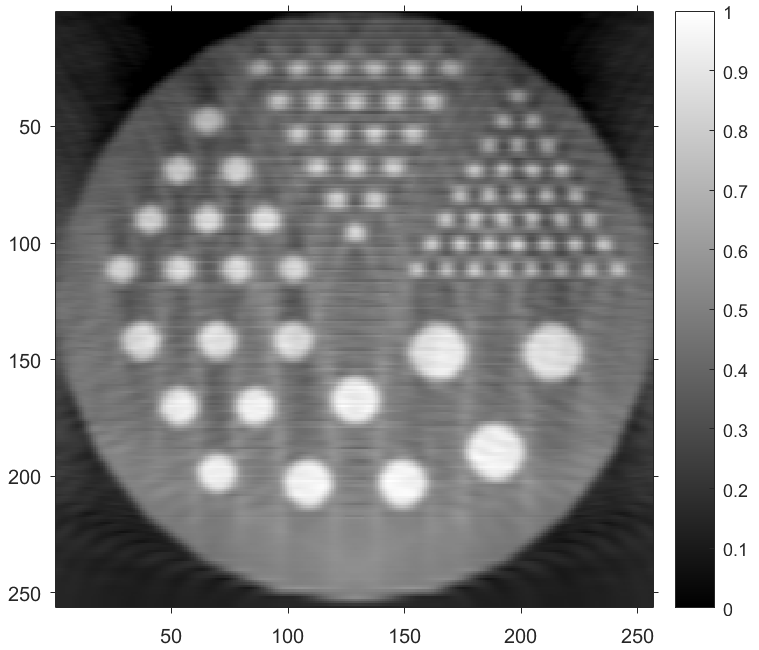}
    \label{fig:der_delta_r_4}
                         } 
    \end{center}
    \caption{Evaluation of the number of scanning circles on reconstruction quality. First row: Reconstruction results of the Derenzo phantom \ref{fig:derenzo_phantom} for $r_{max} = 2N$ \ref{fig:der_rmax_2N}, $3N$ \ref{fig:der_rmax_3N} and $4N$ \ref{fig:der_rmax_4N} where $\Delta_{r}=1$. Second row: Reconstruction results for $\Delta_{r} = 1$ \ref{fig:der_delta_r_1}, $2$ \ref{fig:der_delta_r_2} and $4$ \ref{fig:der_delta_r_4} detector per unit length and $r_{max} = 3N$ remains constant.}
  \label{fig:variying_rmax_and_Delta_r}
\end{figure*}


\subsubsection{Discussions}
The above simulations results show some interesting issues for the CST modality and the reconstruction quality that can be expected with such a system. First, this system is able to scan objects whose depth is largely oversized relative to the source-detector distance. However, this seems to be counterbalanced by the need for a consequent length for the source and detector linear paths, since sufficient reconstruction results appear for source and detector paths of length size six times greater than the object size. The necessary number of scanning circle arcs is also very important since, if we relate the values of $r_{max}$ with the scattering angles, the reconstruction quality is largely improved when $r_{max}$ is high whereas the angular distance between two values of $r_{max}$ is very small.
Considering a larger distance between the source and the detector paths for the system may reduce partially the required amount of projection data. 

Moreover, the geometry offers a sufficient reconstruction quality for every tangent with arbitrary slopes. However, one can notice that vertical slopes are slightly less well reconstructed than the other ones. This can be seen if we pay carefully attention to the left and right sides of the reconstruction. This may be problematic if the object to scan is essentially made of vertical features. One way to avoid this is to perform additional scans by rotating the object, if possible.

Some artefacts that look like shadows around the circles which compose the object remains clearly visible. The issue of artefacts in CST has already been addressed in different manners, for instance in \cite{webber2020microlocal}, where microlocal analysis was employed to alleviate artefacts in reconstructions from limited data. Moreover, in \cite{Webber2020joint_recons}, a penalized iterative algorithm was developed for a mixed modality. Regarding our approach, it can be combined in a pipeline with post-processing stages based on machine learning, as we did in \cite{ayad21} for limited data issues in classical computed tomography. Some work is on the way.

\section{Concluding remarks}
In this article, we proposed a new reconstruction algorithm for a recently proposed CST modality with translational geometry. The algorithm can be numerically implemented efficiently and reconstructions exhibit good quality. A quantitative study of the required data has also been carried out. This study proved the ability for such a CST system to reconstruct larger objects than the system itself. This advantage is moderated by the fact that it is necessary to have an important amount of source-detector positions, which can make the acquisition time longer. An interesting issue of this work concerns a reconstruction algorithm for the three-dimensional extension of this system with planar paths for the source and the detector.

\section{Acknowledgements}
We would like to thank Prof. T. T. Truong for stimulating discussions. 

C. Tarpau research work is supported by grants from R\'egion Île-de-France (in Mathematics and Innovation) 2018-2021 and LabEx MME-DII (Mod\`eles Math\'ematiques et \'Economiques de la Dynamique, de l'Incertitude et des Interactions) (No. ANR-11-LBX-0023-01).

 J. Cebeiro research work is supported by a postdoctoral grant from the University of San Mart\'in. He is also partially supported by SOARD-AFOSR (grant number FA9550-18-1-0523).
 
\bibliographystyle{IEEEtran}
{\footnotesize
\bibliography{IEEEabrv,biblio}}

\begin{thebibliography}{10}
\providecommand{\url}[1]{#1}
\csname url@samestyle\endcsname
\providecommand{\newblock}{\relax}
\providecommand{\bibinfo}[2]{#2}
\providecommand{\BIBentrySTDinterwordspacing}{\spaceskip=0pt\relax}
\providecommand{\BIBentryALTinterwordstretchfactor}{4}
\providecommand{\BIBentryALTinterwordspacing}{\spaceskip=\fontdimen2\font plus
\BIBentryALTinterwordstretchfactor\fontdimen3\font minus
  \fontdimen4\font\relax}
\providecommand{\BIBforeignlanguage}[2]{{%
\expandafter\ifx\csname l@#1\endcsname\relax
\typeout{** WARNING: IEEEtran.bst: No hyphenation pattern has been}%
\typeout{** loaded for the language `#1'. Using the pattern for}%
\typeout{** the default language instead.}%
\else
\language=\csname l@#1\endcsname
\fi
#2}}
\providecommand{\BIBdecl}{\relax}
\BIBdecl

\bibitem{lale1959examination}
P.~Lale, ``The examination of internal tissues, using gamma-ray scatter with a
  possible extension to megavoltage radiography,'' \emph{Physics in Medicine \&
  Biology}, vol.~4, no.~2, p. 159, 1959.

\bibitem{clarke1969compton}
R.~Clarke and G.~Van~Dyk, ``Compton-scattered gamma rays in diagnostic
  radiography,'' in \emph{Medical Radioisotope Scintigraphy. VI Proceedings of
  a Symposium on Medical Radioisotope Scintigraphy}, 1969.

\bibitem{farmer1971new}
F.~Farmer and M.~P. Collins, ``A new approach to the determination of
  anatomical cross-sections of the body by compton scattering of gamma-rays,''
  \emph{Physics in Medicine \& Biology}, vol.~16, no.~4, p. 577, 1971.

\bibitem{redler2018compton}
G.~Redler, K.~C. Jones, A.~Templeton, D.~Bernard, J.~Turian, and J.~C. Chu,
  ``Compton scatter imaging: A promising modality for image guidance in lung
  stereotactic body radiation therapy,'' \emph{Medical physics}, vol.~45,
  no.~3, pp. 1233--1240, 2018.

\bibitem{jones2018characterization}
K.~C. Jones, G.~Redler, A.~Templeton, D.~Bernard, J.~V. Turian, and J.~C. Chu,
  ``Characterization of {C}ompton-scatter imaging with an analytical simulation
  method,'' \emph{Physics in Medicine \& Biology}, vol.~63, no.~2, p. 025016,
  2018.

\bibitem{gautam83}
S.~Gautam, F.~Hopkins, R.~Klinksiek, and I.~Morgan, ``Compton interaction
  tomography {I}. {F}easibility studies for applications in earthquake
  engineering,'' \emph{IEEE Transactions on Nuclear Science}, vol.~30, no.~2,
  pp. 1680--1684, 1983.

\bibitem{prado2017three}
P.~G. Prado, M.~K. Nguyen, L.~Dumas, and S.~X. Cohen, ``Three-dimensional
  imaging of flat natural and cultural heritage objects by a {C}ompton
  scattering modality,'' \emph{Journal of Electronic Imaging}, vol.~26, no.~1,
  p. 011026, 2017.

\bibitem{harding2010compton}
G.~Harding and E.~Harding, ``Compton scatter imaging: A tool for historical
  exploration,'' \emph{Applied Radiation and Isotopes}, vol.~68, no.~6, pp.
  993--1005, 2010.

\bibitem{hussein2005use}
E.~M. Hussein, M.~Desrosiers, and E.~J. Waller, ``On the use of radiation
  scattering for the detection of landmines,'' \emph{Radiation Physics and
  Chemistry}, vol.~73, no.~1, pp. 7--19, 2005.

\bibitem{cruvinel2006compton}
P.~E. Cruvinel and F.~A. Balogun, ``Compton scattering tomography for
  agricultural measurements,'' \emph{Engenharia Agricola}, vol.~26, no.~1, pp.
  151--160, 2006.

\bibitem{livre}
T.~T. Truong and M.~K. Nguyen, ``Recent developments on {C}ompton scatter
  tomography: theory and numerical simulations,'' in \emph{Numerical
  Simulation-From Theory to Industry}.\hskip 1em plus 0.5em minus 0.4em\relax
  IntechOpen, 2012.

\bibitem{nguyen2006imagerie}
M.~K. Nguyen and T.~T. Truong, \emph{Imagerie par rayonnement gamma
  diffus{\'e}}.\hskip 1em plus 0.5em minus 0.4em\relax Herm{\`e}s Science,
  2006.

\bibitem{radon17}
J.~Radon, ``{\"U}ber die {B}estimmung von {F}unktionen durch ihre
  {I}ntegralwerte l\"angs gewisser {M}annigfaltigkeiten,'' \emph{Akad. Wiss.},
  vol.~69, pp. 262--277, 1917.

\bibitem{cormack63}
A.~M. Cormack, ``Representation of a function by its line integrals, with some
  radiological applications,'' \emph{Journal of Applied Physics}, vol.~34,
  no.~9, pp. 2722--2727, 1963.

\bibitem{norton94}
S.~J. Norton, ``Compton scattering tomography,'' \emph{Journal of applied
  physics}, vol.~76, no.~4, pp. 2007--2015, 1994.

\bibitem{nguyen10}
M.~K. Nguyen and T.~T. Truong, ``Inversion of a new circular-arc {R}adon
  transform for {C}ompton scattering tomography,'' \emph{Inverse Problems},
  vol.~26, no.~6, p. 065005, 2010.

\bibitem{num_inv}
G.~Rigaud, M.~K. Nguyen, and A.~K. Louis, ``Novel numerical inversions of two
  circular-arc {R}adon transforms in {C}ompton scattering tomography,''
  \emph{Inverse Problems in Science and Engineering}, vol.~20, no.~6, pp.
  809--839, 2012.

\bibitem{rigaudIEEE13}
G.~Rigaud, R.~R{\'e}gnier, M.~K. Nguyen, and H.~Zaidi, ``Combined modalities of
  {C}ompton scattering tomography,'' \emph{IEEE Transactions on Nuclear
  Science}, vol.~60, no.~3, pp. 1570--1577, 2013.

\bibitem{truong2011radon}
T.~T. Truong and M.~K. Nguyen, ``Radon transforms on generalized {C}ormack's
  curves and a new {C}ompton scatter tomography modality,'' \emph{Inverse
  Problems}, vol.~27, no.~12, p. 125001, 2011.

\bibitem{tarpau19trpms}
C.~Tarpau, J.~Cebeiro, M.~Morvidone, and M.~K. Nguyen, ``A new concept of
  {C}ompton {S}cattering tomography and the development of the corresponding
  circular {R}adon transform,'' \emph{IEEE Transactions on Radiation and Plasma
  Medical Sciences}, vol. (accepted for publication), 2019,
  [10.1109/TRPMS.2019.2943555].

\bibitem{tarpau2020compton}
C.~Tarpau and M.~K. Nguyen, ``Compton scattering imaging system with two
  scanning configurations,'' \emph{Journal of Electronic Imaging}, vol.~29,
  no.~1, p. 013005, 2020.

\bibitem{cebeiro18}
J.~Cebeiro, M.~K. Nguyen, M.~Morvidone, and C.~Tarpau, ``An interior {C}ompton
  {S}catter {T}omography,'' in \emph{25th IEEE Nuclear Science Symposium and
  Medical Imaging Conference 2018 (IEEE NSS/MIC'18)}, Sydney, Australia, Nov.
  2018.

\bibitem{rigaud_rotation_free2017}
\BIBentryALTinterwordspacing
G.~Rigaud, ``Compton {S}cattering {T}omography: {F}eature {R}econstruction and
  {R}otation-{F}ree {M}odality,'' \emph{SIAM Journal on Imaging Sciences},
  vol.~10, no.~4, pp. 2217--2249, 2017. [Online]. Available:
  \url{https://doi.org/10.1137/17M1120105}
\BIBentrySTDinterwordspacing

\bibitem{tarpau2020tci}
C.~{Tarpau}, J.~{Cebeiro}, M.~K. {Nguyen}, G.~{Rollet}, and M.~A. {Morvidone},
  ``Analytic inversion of a {R}adon transform on double circular arcs with
  applications in {C}ompton {S}cattering {T}omography,'' \emph{IEEE
  Transactions on Computational Imaging}, vol.~6, pp. 958--967, 2020.

\bibitem{webber2020compton}
J.~Webber and E.~L. Miller, ``Compton scattering tomography in translational
  geometries,'' \emph{Inverse Problems}, vol.~36, no.~2, p. 025007, 2020.

\bibitem{webber2018three}
J.~W. Webber and W.~R. Lionheart, ``Three dimensional {C}ompton scattering
  tomography,'' \emph{Inverse Problems}, vol.~34, no.~8, p. 084001, 2018.

\bibitem{rigaud20183d}
G.~Rigaud and B.~N. Hahn, ``3{D} {C}ompton scattering imaging and contour
  reconstruction for a class of {R}adon transforms,'' \emph{Inverse Problems},
  vol.~34, no.~7, p. 075004, 2018.

\bibitem{cebeiro2021ip}
J.~Cebeiro, C.~Tarpau, M.~A. Morvidone, D.~Rubio, and M.~K. Nguyen, ``On a
  three dimensional {C}ompton scattering tomography system with fixed source,''
  \emph{Inverse Problems}, vol.~37, no.~5, p. 054001, 2021.

\bibitem{webber2017microlocal}
\BIBentryALTinterwordspacing
J.~W. Webber and S.~Holman, ``Microlocal analysis of a spindle transform,''
  \emph{AIMS Inverse Problems and Imaging}, vol.~13, no.~2, pp. 231--261, 2019.
  [Online]. Available:
  \url{http://aimsciences.org//article/id/7ad5560c-e076-4384-9e9d-1dab4121da6d}
\BIBentrySTDinterwordspacing

\bibitem{cebeiro2017new}
J.~Cebeiro, M.~K. Nguyen, M.~Morvidone, and A.~Noumow{\'e}, ``New
  “improved” {C}ompton scatter tomography modality for investigative
  imaging of one-sided large objects,'' \emph{Inverse Problems in Science and
  Engineering}, vol.~25, no.~11, pp. 1676--1696, 2017.

\bibitem{truong2019compton}
T.~Truong and M.~Nguyen, ``Compton scatter tomography in annular domains,''
  \emph{Inverse Problems}, vol.~35, no.~5, p. 054005, 2019.

\bibitem{Gradshteyn:1702455}
\BIBentryALTinterwordspacing
I.~S. Gradshteyn, I.~M. Ryzhik, D.~Zwillinger, and V.~Moll, \emph{{Table of
  integrals, series, and products; 8th ed.}}\hskip 1em plus 0.5em minus
  0.4em\relax Amsterdam: Academic Press, Sep 2014. [Online]. Available:
  \url{https://cds.cern.ch/record/1702455}
\BIBentrySTDinterwordspacing

\bibitem{RB1990}
R.~N. Bracewell, ``Numerical transforms,'' \emph{Science}, vol. 248, no. II
  May, pp. 697--704, 1990.

\bibitem{bateman_table_vol1}
H.~Bateman, \emph{Tables of Integral Transforms}.\hskip 1em plus 0.5em minus
  0.4em\relax New York: McGraw-Hill Book Compagny, 1954, vol.~1.

\bibitem{truong_symmetry}
\BIBentryALTinterwordspacing
T.~T. Truong, ``Function reconstruction from reflection symmetric radon data,''
  \emph{Symmetry}, vol.~12, no.~6, 2020. [Online]. Available:
  \url{https://www.mdpi.com/2073-8994/12/6/956}
\BIBentrySTDinterwordspacing

\bibitem{webber2020microlocal}
J.~W. Webber and E.~T. Quinto, ``Microlocal analysis of a compton tomography
  problem,'' \emph{SIAM Journal on Imaging Sciences}, vol.~13, no.~2, pp.
  746--774, 2020.

\bibitem{Webber2020joint_recons}
\BIBentryALTinterwordspacing
J.~W. Webber, E.~T. Quinto, and E.~L. Miller, ``A joint reconstruction and
  lambda tomography regularization technique for energy-resolved x-ray
  imaging,'' vol.~36, no.~7, p. 074002, jul 2020. [Online]. Available:
  \url{https://doi.org/10.1088/1361-6420/ab8f82}
\BIBentrySTDinterwordspacing

\bibitem{ayad21}
I.~Ayad, C.~Tarpau, M.~K. Nguyen, and N.~S. Vu, ``Deep morphological
  network-based artifact suppression for limited-angle tomography,'' in
  \emph{Proceedings of the 25th International Conference on Image Processing,
  Computer Vision and Pattern Recognition (IPCV’21)}, Las Vegas, United
  States, Jul. 2021.

\end{thebibliography}

\end{document}